\documentstyle[amssymb,amsfonts]{amsart}

\newenvironment{proof}{\noindent {\bf Proof} }{\endprf\par}
\def \endprf{\hfill  {\vrule height6pt width6pt depth0pt}\medskip}

\theoremstyle{plain}
  \newtheorem{theorem}[subsection]{Theorem}
  
  \newtheorem{proposition}[subsection]{Proposition}
  \newtheorem{lemma}[subsection]{Lemma}
  \newtheorem{corollary}[subsection]{Corollary}

\theoremstyle{remark}

\theoremstyle{definition}

\include{psfig}

\begin{document}

\title[Algebraic methods in discrete analogs of the
 Kakeya problem]{Algebraic methods in discrete analogs of the Kakeya problem}

\author{Larry Guth}
\address{Department of Mathematics, University of Toronto}
\email{ lguth@@math.toronto.edu}

\author{Nets Hawk Katz}
\address{Department of Mathematics, Indiana University, Bloomington}
\email{nhkatz@@indiana.edu}

\vspace{-0.3in}

\begin{abstract} We prove the joints conjecture, showing that for
any $N$ lines in ${\Bbb R}^3$, there are at most $O(N^{{3 \over
2}})$ points at which 3 lines intersect non-coplanarly. We also
prove a conjecture of Bourgain showing that given $N^2$ lines in
${\Bbb R}^3$ so that no $N$ lines lie in the same plane and so
that each line intersects a set $P$ of points in at least $N$
points then the cardinality of the set of points is
$\Omega(N^3)$.  Both our proofs are adaptations of Dvir's
argument for the finite field Kakeya problem.
\end{abstract}

\maketitle

\section{Introduction}

Various authors have considered the joints problem. It asks, given
$N$ lines in space, how many ``joints" can the lines form, where
a joint is defined as a point where three lines with linearly
independent directions intersect. Obviously given a $\sqrt{ {N
\over 3}} \times \sqrt{ {N \over 3}} \times \sqrt{ {N \over 3}}$
cube in the integer lattice, we get $N$ lines with ${N^{{3 \over
2}} \over 3 \sqrt{3}}$ joints by simply taking all lines in
coordinate directions which intersect the cube and the lattice.
The joints problem is to prove:

\begin{theorem} \label{joints} Any set of $N$ lines in ${\Bbb
R}^3$ form at most $O(N^{{3 \over 2}})$ joints. \end{theorem}

The previous best bound in the joints problem is due to Feldman
and Sharir \cite{FS}, who proved that the number of joints is
$O(N^{1.6232})$. Earlier bounds were reported in \cite{SW},
\cite{S}, and \cite{CEGPSSS}.
Bennett,Carbery, and Tao obtained a result conditioned on the
angles at the joints in \cite{BCT}.

At AIM in 2004, Bourgain conjectured the following:

\begin{theorem} \label{Bourgain} Let $L$ be a set of $N^2$ lines
in ${\mathbb{R}}^3$ and let $P$ be a set of points in
${\mathbb{R}}^3$. Suppose no more than $N$ lines of $L$ lie in the
same plane and suppose the each line of $L$ contains at least $N$
points of $P$. Then $|P|=\Omega(N^3)$.
\end{theorem}

The previous best bound in Bourgain's incidence problem is due to
Solymosi and Toth \cite{ST}, who proved that the number of points is
$\Omega(N^{11/4})$.

What both conjectures have in common is that they are discrete
models of the Kakeya problem. Work of Sharir on joints helped
inspire Schlag's program on Kakeya, see \cite{Sch}. Bourgain's conjecture was
posed with the analogy to Kakeya in mind.

In 2008, Dvir (\cite{D}) solved the finite field version of the Kakeya
problem. His technique was to study the properties of a
polynomial which vanishes on the Kakeya set. We adapt this
technique, proving the above theorems through a study of a
polynomial which vanishes on the point sets in question. 

The main idea of both proofs is as follows. We find a polynomial of as low degree
as possible which vanishes on the set of joints ({\it resp.} points) in question.
We factorize the polynomial to irreducibles and find an irreducible factor
vanishing on many joints(points). That irreducible factor will also vanish on many
lines. In the event that the gradient, too, vanishes on many lines, a variant of
Bezout's theorem leads to a contradiction by reducing our irreducible. Otherwise 
at each point where such lines intersect, all the lines lie in the tangent plane.
In the joints problem this leads to an immediate contradiction, since intersections
must be non-coplanar. In the Bourgain problem, this leads to many flat points
of the zero set of the irreducible, which force that zero set to be a plane. This
contradicts the hypothesis of fewer than $N$ lines in a plane. The idea of planiness,
that in Kakeya problems, lines at a given point of intersection lie mainly in a plane,
seems first to have arisen in the work of the second author with Laba and Tao \cite{KLT}.
The idea that this plane is the tangent space to a polynomial vanishing on the set
comes from the work of the first author on the endpoint multilinear Kakeya problem
\cite{G}. 

We have tried to minimize the amount of algebra background needed for the paper.  The
small amount of algebra we use is summarized in the next section with references.
The main ingredient is Bezout's theorem.  In order to minimize the algebra, we focus
on irreducible polynomials, and we use pigeonhole estimates to locate an irreducible
polynomial that vanishes on many joints or points.  It is also possible to give a proof
using reducible polynomials.  Such a proof would need less pigeonhole estimates, but
it would require more algebra.

{\bf Acknowledgements:} The second author would like to thank the University of Toronto
and the Fields Institute for their hospitality which allowed this collaboration to
take place. He would also like to thank his colleague Michael Larsen for a useful
conversation about resultants. He is partially supported by NSF grant DMS 0653763.

\vskip.125in

\section{Algebraic Preliminaries}

In this section, we bring together various algebraic facts that
we shall need. Good references are the books of Cox, Little, and
O'Shea. (\cite{CLO},\cite{CLOU})

We recall a fundamental object, the {\it resultant} of two polynomials. Given
$f$ and $g$ elements of ${\Bbb C}[x]$ having degree $l$ and $m$ respectively,
and given as
$$f(x)=a_l x^l + a_{l-1} x^{l-1} + \dots + a_0,$$
and
$$g(x)=b_m x^m + b_{m-1} x^{m-1} + \dots + b_0,$$
we define the resultant of $f$ and $g$, namely $Res(f,g)$ as the determinant of
the $l+m  \times  l+m$ matrix whose coefficients $c_{ij}$ satisfy
$c_{ij} =a_{j-i}$ if $1 \leq i \leq m$ and $i \leq j \leq i+l$, satisfy
$c_{ij}= b_{j-i +m}$ if $m+1 \leq i \leq m+l$ and $i-m \leq j \leq i-m+l$
and are equal to zero otherwise.

The columns of the matrix $c_{ij}$ represent coefficients of the polynomial
$f$ multiplied by $x^j$ where $j$ runs from 0 to $m-1$ and the
coefficents of the polynomial $g$ multiplied by $x^k$  where $k$ runs
from 0 to $l-1$. The resultant tests whether this set of polynomials is linearly
independent. Linear independence fails exactly when $f$ and $g$ have a common
factor. (The resultant was first defined in this way by Sylvester.) We bring together
some basic properties of the resultant following \cite{CLOU}.

Now suppose instead that we work with polynomials $f,g \in {\Bbb C}[x_1, \dots , x_n]$.
We may view them as polynomials in $x_1$ with coefficients that are polynomials
in $x_2,\dots, x_n$. Then we denote the resultant, a polynomial in $x_2,\dots,x_n$ as
$Res(f,g;x_1)$. In fact, computing resultants is all we need to do in order to
determine whether polynomials in several variables have a common factor.

\begin{theorem} \label{factor} Let $f,g \in {\Bbb C}[x_1, \dots , x_n]$ and suppose
that both $f$ and $g$ have positive degree when viewed as polynomials in $x_1$ then
$f$ and $g$ have a common factor if and only if $Res(f,g;x_1)$ is identically zero.
\end{theorem}

Theorem \ref{factor} is \S 3.6 proposition 1 (ii) in \cite{CLOU}.

\begin{proposition} \label{Bezoutmain}  Let $f$ and $g$ be elements of ${\Bbb C}[x_1,x_2]$
and suppose that $f$ and $g$ have degrees $l$ and $m$ respectively.  Furthermore,
assume that $f$ has degree $l$ in $x_1$ and $g$ has degree $m$ in $x_1$.  Then $Res(f,g;x_1)$ is 
a polynomial of $x_2$ of degree at most $lm$. \end{proposition}

\begin{proof} Given two polynomials of one variable, 
$$f(x)=(x-r_1) (x-r_2) \dots  (x-r_l),$$
and 
$$g(x)=(x-s_1) (x-s_2) \dots (x-s_m),$$
we have that
$$Res(f,g)=\prod_{j=1}^l \prod_{k=1}^m  r_j-s_k.$$
The coefficient $a_i$ of $f$ is a symmetric polynomial in the roots of $f$ which is homogeneous
of degree $l-i$.  Similarly, $b_i$ is a homogenenous polynomial of degree $m-i$ in
the roots of $g$.  We therefore assign the variable $a_i$ a degree $l-i$ and the
variable $b_i$ the degree $m-i$.  With respect to these degrees, the resultant 
$Res(f,g)$ is a homogeneous polynomial of degree $l m$.  
On the other hand, the coefficient $a_i$ is a polynomial in
$x_2$ of degree at most $l-i$ and the coefficient $b_i$ is a polynomial in $x_2$ of
degree at most $m-i$.  Therefore, $Res(f,g; x_1)$ is a polynomial of $x_2$ of degree
at most $l m$. \end{proof}

Proposition \ref{Bezoutmain} is the main point in the proof of the celebrated Bezout
theorem.

\begin{corollary} \label{Bezout} (Bezout's theorem) Let $f$ and $g$ be elements of 
${\Bbb C}[x_1,x_2]$ and suppose that $f$ and $g$ have positive 
degrees $l$ and $m$ respectively.
Suppose there are more than $lm$ points of ${\Bbb C}^2$ where $f$ and $g$ both vanish. Then
$f$ and $g$ have a common factor. \end{corollary}

\begin{proof} We begin by changing coordinates.  We may choose $x_1$ and $x_2$
so that $f$ has degree $l$ in $x_1$ and $g$ has degree $m$ in $x_1$.  We may also guarantee
that there are more than $l m$ distinct values of $x_2$ among the points where $f$ and $g$ both
vanish.  A generic choice of coordinates accomplishes these goals.  By Proposition \ref{Bezoutmain},
we see that $Res(f,g;x_1)$ is a polynomial of $x_2$ with degree at most $l m$.  Since it vanishes
at more than $l m$ points, it must vanish identically.  An application of Theorem \ref{factor}
completes the proof. \end{proof}

We shall need a small generalization of corollary \ref{Bezout} that works in ${\Bbb C}^3$ when two polynomial
vanish simultaneously on many lines.

\begin{corollary} \label{BezoutlinesC} Let $f$ and $g$ be elements of 
${\Bbb C}[x_1,x_2,x_3]$ and suppose that $f$ and $g$ have positive 
degrees $l$ and $m$ respectively. Suppose there are more than $lm$ lines on which
$f$ and $g$ simultaneously vanish identically. Then $f$ and $g$ have
a common factor. \end{corollary}

\begin{proof} Without loss of generality, we may choose $x_1$ so that $f$ and $g$ have
positive degree in $x_1$ and $x_3$ so that the $x_3=0$ plane is transverse to at least
$lm+1$ of the lines of vanishing. Then fixing $x_3$ and apply Bezout's theorem
and theorem \ref{factor}, we get that $Res(f,g;x_1)$ vanishes identically as a function
of $x_2$. Since this happens for all values of $x_3$, we have that $Res(f,g;x_1)$
vanishes identically and therefore applying theorem \ref{factor}, we get the desired
result. \end{proof}

Finally we prove the real analog of corollary \ref{BezoutlinesC}.  The result below is the one
we apply in the proof of our theorems.

\begin{corollary} \label{Bezoutlines} Let $f$ and $g$ be elements of ${\Bbb R}[x_1,
x_2, x_3]$, and suppose that $f$ and $g$ have positive degrees $l$ and $m$
respectively.  Suppose that there are more than $l m$ lines on which $f$ and
$g$ simultaneously vanish identically.  Then $f$ and $g$ have a common factor.
\end{corollary}

\begin{proof} We can think of $f$ and $g$ as elements of ${\Bbb C}[x_1, x_2, x_3]$, and
they must vanish on more than $l m$ complex lines in ${\Bbb C}^3$.  By corollary \ref{BezoutlinesC},
$f$ and $g$ must have a common factor $h$ in ${\Bbb C}[x_1, x_2, x_3]$.  We can assume
$h$ is irreducible.  A priori,
the polynomial $h$ may or may not be real.  But, if $h$ is non-real, then the irreducible
factorization of $f$ must contain both $h$ and $\bar h$.  Hence $f$ is divisible by
the real polynomial $h \bar h$.  Similarly, $g$ is divisible by $h \bar h$. \end{proof}

We take this moment to state an additional algebraic proposition which we will
use in what follows.

\begin{proposition} \label{yeswecan} Let $X$ be a set of $N$ points in ${\Bbb R}^3$.
Then there is a nontrivial polynomial in ${\Bbb R}[x_1,x_2,x_3]$ ({\it a fortiori} in 
${\Bbb C}[x_1,x_2,x_3]$) which vanishes on every point of $X$ of degree less than
$C N^{{1 \over 3}},$ with $C$ a universal constant.
\end{proposition}

\begin{proof} A polynomial of three variables and degree $d$ has 
${(d+3) (d+2) (d+1) \over 6}$ coefficients. Requiring that a polynomial
vanish at a point gives a homogeneous linear equation for the coefficients.
Underdetermined systems of homogeneous linear equations always have nontrivial
solutions. \end{proof}

\vskip.125in

\section{Geometric Preliminaries}

In this section, we will recall some basic facts of the extrinsic geometry
of irreducible algebraic varieties in ${\Bbb R}^3$.

We let $p$ be a nontrivial irreducible polynomial on ${\Bbb R}^3$ of degree
$d>0$. We consider
$$S=\{(x,y,z): p(x,y,z) = 0\}.$$

We say that a point $a \in S$ is {\it critical} if $\nabla p(a)=0$. Otherwise,
we say that $a$ is {\it regular}. (By the implicit function theorem, $S$ is locally
a manifold in a neighborhood of a regular point.) We say a line $l$ is {\it critical}
if $l \subset S$ and every point of $l$ is critical.

\begin{proposition} \label{nottoocritical} The set $S$ contains no more
than $d(d-1)$ critical lines. \end{proposition}

\begin{proof} Suppose not. We apply corollary \ref{Bezoutlines} to $p$ and a nontrivial
component of $\nabla p$. This contradicts the irreducibility of $p$. \end{proof}

Next, we begin to investigate regular points of $S$. We immediately get the following.

\begin{proposition} \label{tangent} Let $a$ be a regular point of $S$. Let $l$
be a line contained in $S$ which passes through $a$. Then $l \subset T_a S,$ where
$T_a S$ is the tangent plane to $S$ at $a$. \end{proposition}

Let $a$ be a regular point of $S$. We would like to investigate the extrinsic
curvature of $S$ at $a$. That is, we want to understand infinitesimally how the
direction of $\nabla p$ is changing in a neighborhood of $a$ in $S$. We define
$${\bf II}(p)(a)=\{ \nabla_{\nabla p \times e_j} \nabla p \times \nabla p \}_{j=1,2,3},$$
where $\times$ is the cross product of vectors and $e_1,e_2,e_3$ are the standard
basis vectors in ${\Bbb R}^3$. Clearly ${\bf II}(p)$ is a set of three vectors. Thus
it has nine components. Each of the components is a polynomial of degree no
more than $3d-4$.

We will refer to ${\bf II}(p)$ as the algebraic second fundamental form of $S$. (The
geometric second fundamental form is a quadratic form on $T_p S$ obtained by 
differentiating the unit normal vector to $S$ along $S$. However since the algebraic 
fundamental form measures the normal component of the change of $\nabla p$ along
three directions which span $T_p S$, it is easy to see that for any regular point $a$,
all the components of ${\bf II}(p)(a)$ vanish if and only if the second fundamental
form of $S$ vanishes.

We say a regular point $a$ of $S$ is {\it flat} if all the components of
${\bf II}(p)(a)$ vanish. We give a sufficient condition for a regular point $a$ to be
flat.

\begin{lemma} \label{flatpoints}  Let $a$ be a regular point of $S$. Suppose
that $S$ contains three distinct lines all of which intersect at $a$, then
$a$ is a flat point. \end{lemma} 

\begin{proof} The quadratic surface which most closely
approximates $S$ at $a$ contains the three lines. 
(This is because the Taylor series of $p$ along the
lines vanishes.) However so does $T_a S$. Thus since $3>2$ by corollary \ref{Bezoutlines}
the quadratic surface most closely approximating $S$ must contain $T_a S$. Therefore,
the second fundamental form of $S$ vanishes at $a$. \end{proof}

We say that a line $l$ is flat if it is contained in $S$, it is not critical,
and every regular point of $l$ is flat.

\begin{corollary} \label{it'saplane!} Let $p$ be an irreducible polynomial of degree
$d>0$. Let $S$ be the zero set of $p$. Suppose that $S$ contains more than $3d^2-4d$
flat lines, then $S$ is a plane. \end{corollary}

\begin{proof} By corollary \ref{Bezoutlines}, each component of ${\bf II}(p)$
has $p$ as a factor. Therefore the direction of the normal to $S$ at regular points
of $S$ does not vary. Therefore $S$ contains a plane. But $p$ is irreducible so
$S$ is a plane. \end{proof}

\vskip.125in

\section{Analytic preliminaries}

In the following two sections we will prove  ``big oh" results by contradiction.
Thus we will make an assumption involving a large constant $K$. We keep track
of anything that depends on $K$. However, we ignore constants which are independent
of $K$. Thus we write
$$A \lesssim B$$
if $A$ and $B$ are quantities and $C$ is a universal constant.

There are two variants of the pigeonhole principle which we will use frequently.
The first is often referred to as the popularity lemma.

\begin{lemma} \label{popular} Let $(X,Y,E)$ be a bipartite graph with $E$ the edges and
$X$ and $Y$ the two sets of vertices. Suppose that $|E| > \rho |Y|.$ Let
$Y^{\prime}$ be the set of vertices of $Y$ having degree at least $\mu$ and let
$E^{\prime}$ be the set of edges in $E$ between $Y^{\prime}$ and $X$
Then
$$|E^{\prime}| > (\rho-\mu)|Y|.$$
\end{lemma}

\begin{proof} The vertices in $Y \backslash Y^{\prime}$ are incident to at most
$\mu |Y|$ edges. \end{proof}

We now give the second which we'll refer to freely as the pigeonhole principle.

\begin{lemma} \label{pigeonhole} Let $x_1,\dots , x_m$ be positive quantities and
$y_1,\dots, y_m$ positive quantities, then there is  an integer $1 \leq k \leq m$
so that
$${x_k \over y_k} \geq {\sum_{j=1}^m x_j \over \sum_{j=1}^m y_j}.$$
\end{lemma}

\begin{proof} Suppose not. Let $\rho={\sum_{j=1}^m x_j \over \sum_{j=1}^m y_j}.$
Then $x_k < \rho y_k$ for all $k$. Thus
$$\sum_{j=1}^m x_j < \rho \sum_{j=1}^m y_j,$$
which is a contradiction. \end{proof}

\section{Joints Problem}

In this section, we prove theorem \ref{joints}.

We suppose that we are given a set of lines $L$ of cardinality $N$. Let $J$ be the
set of joints determined by $L$. We suppose 
$$|J| \geq K N^{{3 \over 2}},$$
with $K$ a large, but universal, constant.

We create a bipartite, three-colored graph $(L,J,R,G,B)$ between the set of lines and
the set of joints. Each joint is incident to at least three noncoplanar lines. For
each joint, we arbitrarily color one of the incidences red, one green, and one blue.
The sets $R,G,$ and $B$ consist of, respectively, the red, green, and blue incidences.

We will now refine the sets slightly. We let $L_R$ be the set of all lines
which have at least ${K \over 1000} N^{{1 \over 2}}$ incidences in $R$. We 
let $J_R$ be the set of joints having a red incidence with a line of $L_R$.
By lemma \ref{popular},
$$|J_R| \geq {999 \over 1000} |J|.$$
Now we let $L_G$ and $L_B$ be those lines having respectively at least 
${K \over 1000} N^{{1 \over 2}}$ green or blue incidences with joints in $J_R$.
We let $J^{\prime}$ denote that set of joints which has red, green, and blue
incidences with lines in $L_G$ and $L_B$ and by lemma \ref{popular},
it is easy to show that
$$|J^{\prime}| \geq {99 \over 100} |J|.$$

Our goal now is to produce a polynomial of relatively low degree vanishing on
all the points of $J^{\prime}$. (Any degree which is substantially lower
than $N^{{1 \over 2}}$ will suffice. We say a line $l$ of $L_G$ or $L_B$ meets
a line $l^{\prime}$ of $L_R$ if $l \cap l^{\prime}$ is a joint of $J_R$. Each
line of $L_G$ and each line of $L_B$ meets 
at least ${K \over 1000} N^{{1 \over 2}}$ lines of $L_R$. We now take a random
subset $L_R^{\prime}$ of the lines of $L_R$, picking each line with probability
${1 \over K}$. 

By the law of large numbers, with positive probability, the following events occur:
Each line of $L_G$ and $L_B$ meets at least ${1 \over 2000}N^{{1 \over 2}}$ 
lines of $L_R^{\prime}$
and
$$|L_R^{\prime}| \leq {2N \over K}.$$

We make a set of points $S$ by selecting ${1 \over 2} N^{{1 \over 2}}$ points on each 
line of $L_R^{\prime}$. Then
$$|S| \leq {N^{{3 \over 2}} \over K}.$$
We find a polynomial $p$ which vanishes on all points of $S$. By the estimate
on the size of $S$, we may select $p$ with degree $d$ which is 
$\lesssim {N^{{1 \over 2}} \over K^{{1 \over 3}}}$. With $K$ sufficiently large, this
means that $p$ must vanish on each line of $L_R^{\prime}$ and because of the number
of lines of $L_R^{\prime}$ that each line of $L_G$ and $L_B$ meet, it means
that $p$ must vanish identically on each line of $L_G$ and $L_B$. Therefore,
the polynomial $p$ must vanish on the entire set $J^{\prime}$.

Now, it is not necessarily the case that $p$ is irreducible. Thus we factor $p$
into irreducibles
$$p=\prod_{j=1}^m p_j.$$
We denote the degree of the polynomial $p_j$ by $d_j$ and observe that
$$\sum_{j=1}^m d_j \lesssim {N^{{1 \over 2}} \over K^{{1 \over 3}}}.$$
We let $J_j$ be the subset of $J^{\prime}$ on which $p_j$ vanishes, and we
have
$$\sum_{j=1}^m |J_j| \gtrsim  K N^{{3 \over 2}}.$$
Thus by lemma \ref{pigeonhole}, we find $j$ for which
$$|J_j| \gtrsim K^{{4 \over 3}} N d_j.$$
>From now on, we restrict our attention to $J_j$.

We denote by $L_{R,j}$,$L_{G,j}$, and $L_{B,j}$ those lines in $L_R,L_B,$
and $L_G$ which are incident to at least $d_j+1$ elements of $J_j$,
and we let $J_j^{\prime}$ be those element of $J_j$ incident to a line each from
$L_{R,j}$,$L_{G,j}$, and $L_{B,j}$. With $K$ sufficiently large, we have
$$|J_{j}^{\prime}| \geq {999 \over 1000} |J_j|.$$

We now define $L_{R,j}^{\prime}$,$L_{G,j}^{\prime}$, and $L_{B,j}^{\prime}$ as the
set of lines which are incident to more than $d_j+1$ points of $J_j^{\prime}$.
We define $J_j^{\prime \prime}$ to be the set of joints defined by these lines.
We have
$$|J_j^{\prime \prime}| \geq {99 \over 100} |J_j|.$$

We now break into two cases. In the first case, there are fewer than $d_j^2$ lines
in each of $L_{R,j}^{\prime}$,$L_{G,j}^{\prime}$, and $L_{B,j}^{\prime}$. In this
case, we start over again, having a joints problem with fewer lines and more favorable
exponents than the original.

In the second case, we may assume without loss of generality that  $L_{R,j}^{\prime}$
contains at least $d_j^2$ lines. By the definition of $L_{R,j},L_{G,j},$ and
$L_{B,j}$, the polynomial $p_j$ vanishes identically on each line in these sets.
However, this implies that each point of $J_j^{\prime}$ is a critical point of $p_j$,
because otherwise it would be impossible for $p_j$ to vanish on three, intersecting,
non-coplanar lines. But this implies that each component of the gradient of $p_j$
vanishes at each point of $J_j^{\prime}$. Let $q$ be one of the components of
the gradient which does not vanish identically. Then $q$ has degree at most $d_j-1$.
Thus, it must vanish on every line of $L_{R,j}^{\prime}$. But this is a contradiction
by proposition \ref{nottoocritical}.

\vskip.125in

\section{Bourgain's Incidence Problem}

\vskip.125in

In this section, we prove theorem \ref{Bourgain}. We suppose we have a set of points
$X \subset {\Bbb R}^3$ of cardinality ${N^3 \over K}$, with 
$K$ large to be specified later and a set $L$
of $N^2$ lines so that no $N$ lines lie in the same plane and so that each line $l \in L$
is incident to at least $N$ points of $X$. We may assume in what follows
that each line is incident
to exactly $N$ points by coloring $N$ of its incidences black and only counting black
incidences below.

We say that a point $x \in X$ is {\it valuable} if it is incident to at least
${K \over 1000}$ lines. We define $v(x)$ the value of $x$ to be the number of lines it
is incident to. We let $X_v$ be the set of valuable points.
Clearly
$$|X_v| \leq {1000 N^3 \over K},$$
and by lemma \ref{popular}
$$\sum_{x \in X_v} v(x) \geq {999 N^3 \over 1000}.$$

We now perform some dyadic pigeonholing to uniformize the quantity $v(x)$.
We define $X_j$ to be the set of those $x \in X_v$ so that
$$ {2^{j-1} K \over 1000} \leq v(x) < {2^j K \over 1000}.$$
We define 
$$V_j= \sum_{x \in X_j} v(x).$$
Then note that
$$\sum_{j=1}^{\infty} V_j = \sum_{x \in X_v} v(x) \geq {999 N^3 \over 1000}.$$
Note also that
$$\sum_{j=1}^{\infty} {1 \over j^2} ={\pi^2 \over 6} < 2.$$
Hence, by the pigeonhole principle, we can fix a positive number $j$ so that
$$ V_j \geq {999 N^3 \over 2000 j^2}.$$
(In the argument below, the reader is encouraged to assume that $j=1$, since
this is indeed the worst case.)

>From this we see that
$${N^3 \over K 2^j} \gtrsim  |X_j| \gtrsim {N^3 \over K 2^j j^2}.$$
Next we find a polynomial $p$ which vanishes on every point of $X_j$.
We may choose $p$ to have degree $d$ satisfying
$$d \lesssim {N \over K^{{1 \over 3}} 2^{{j \over 3}}}.$$
The polynomial $p$ need not be irreducible. Thus we factor it
$$p=p_1 p_2 \dots p_m,$$
with $p_k$ having degree $d_k$. We let $X_{j,k}$ be the set of points of $X_j$ on
which $p_k$ vanishes.
Clearly, we have
$$d_1 + d_2 + \dots + d_m \lesssim {N \over K^{{1 \over 3}} 2^{{j \over 3}}},$$
while
$$|X_{j,1}| + |X_{j,2}| + \dots + |X_{j,m}| \gtrsim |X_j| \gtrsim {N^3 \over K 2^j j^2}.$$
Thus by the pigeonhole principle, we can fix a $k$ with 
$$|X_{j,k}| \gtrsim {N^2 d_k \over K^{{2 \over 3}} 2^{{2j \over 3}} j^2}.$$
We let $Y=X_{j,k}$ and by the definition of $X_j$, if $I$ is the number of incidences
between $L$ and $Y$, we have
$$I \gtrsim N^2 d_k K^{{1 \over 3}} 2^{{j \over 3}} j^{-2} >> N^2 d_k.$$
We let $L^{\prime}$ be the set of lines incident to more than $100d_k$ points of $Y$
and let $I^{\prime}$ be the number of incidences between $L^{\prime}$ and $Y$. Then
clearly
$$I^{\prime} \gtrsim I.$$
Note that each line of $L^{\prime}$ is in the zero set of $p_k$. Now let
$Y^{\prime}$ be the set of points of $Y$ incident to more than 3 lines of
$L^{\prime}$. Then each point of $Y^{\prime}$ is either a critical point of
$p_k=0$ or else by lemma \ref{flatpoints}, it must be a flat point of $p_k$.

We let $L^{\prime \prime}$ be the set of lines in $L^{\prime}$ incident to at
least $10 d_k$ points of of $Y^{\prime}$. If $I^{\prime \prime}$ is the number
of incidences between lines of $L^{\prime \prime}$ and points of $Y^{\prime}$, we still
have
$$I^{\prime \prime} \gtrsim N^2 d_k K^{{1 \over 3}} 2^{{j \over 3}} j^{-2}.$$
We let $I_{flat}$ and $I_{crit}$ be the number of those incidences with flat points and
critical points respectively. Note that 
$$I_{crit}+I_{flat} \geq I^{\prime \prime}.$$

There are two cases. In the first case
$$I_{crit} \gtrsim N^2 d_k K^{{1 \over 3}} 2^{{j \over 3}} j^{-2},$$
which means that there are at least $N d_k K^{{1 \over 3}} 2^{{j \over 3}} j^{-2} >> d_k^2$
lines in the surface $p_k=0$ on which are critical. But this is a contradiction in
light of proposition \ref{nottoocritical}.

In the second case
$$I_{flat} \gtrsim N^2 d_k K^{{1 \over 3}} 2^{{j \over 3}} j^{-2},$$
which means that there are at least 
$N d_k K^{{1 \over 3}} 2^{{j \over 3}} j^{-2} >> 3d_k^2$
flat lines in the surface $p_k=0$. In light of corollary
\ref{it'saplane!}
 the surface $p_k=0$ is in fact a plane. But now we have more
than $N$ lines of $L$ lying in a plane which is also a contradiction.

\end{document}